\documentclass[12pt]{amsart}

\usepackage{amsmath,amssymb,amsfonts,amsthm,latexsym,graphicx,multirow,enumerate,hyperref}

\newcommand\A{\mathrm{A}} \newcommand\Aut{\mathrm{Aut}}

\newcommand\D{\mathrm{D}}\newcommand\DGP{\mathrm{DGP}}\newcommand\DP{\mathrm{DP}}

   \newcommand\GL{\mathrm{GL}}  \newcommand\GP{\mathrm{GP}}

       \newcommand\Sy{\mathrm{S}}

\newcommand\ZZ{\mathbb{Z}}

\newtheorem{theorem}{Theorem}[section]
\newtheorem{lemma}[theorem]{Lemma}
\newtheorem{proposition}[theorem]{Proposition}
\newtheorem{corollary}[theorem]{Corollary}

\newtheorem{conjecture}[theorem]{Conjecture}

\theoremstyle{definition}

\usepackage{color}

\definecolor{Blue}{rgb}{0,0,1}
\definecolor{Red}{rgb}{1,0,0}
\definecolor{DarkGreen}{rgb}{0,0.6,0}
\definecolor{DarkYellow}{rgb}{1,1,0.2}
\definecolor{DarkPurple}{rgb}{.6,0,1}

\usepackage{xcolor}
\usepackage[normalem]{ulem}

\def\non{\nonumber}

\begin{document}

\title[Canonical double covers of generalized Petersen graphs]{Canonical double covers of generalized Petersen graphs, and double generalized Petersen graphs}

\author[Qin]{Yan-Li Qin}
\address{School of Statistics\\Capital University of Economics and Business\\Beijing, 100070\\ P. R. China}
\email{yanliqin@bjtu.edu.cn}

\author[Xia]{Binzhou Xia}
\address{School of Mathematics and Statistics\\The University of Melbourne\\Parkville, VIC 3010\\Australia}
\email{binzhoux@unimelb.edu.au}

\author[Zhou]{Sanming Zhou}
\address{School of Mathematics and Statistics\\The University of Melbourne\\Parkville, VIC 3010\\Australia}
\email{sanming@unimelb.edu.au}

\maketitle

\begin{abstract}
The canonical double cover $\D(\Gamma)$ of a graph $\Gamma$ is the direct product of $\Gamma$ and $K_2$. If $\Aut(\D(\Gamma))\cong\Aut(\Gamma)\times\ZZ_2$ then $\Gamma$ is called stable; otherwise $\Gamma$ is called unstable. An unstable graph is said to be nontrivially unstable if it is connected, non-bipartite and no two vertices have the same neighborhood. In 2008 Wilson conjectured that, if the generalized Petersen graph $\GP(n,k)$ is nontrivially unstable, then both $n$ and $k$ are even, and either $n/2$ is odd and $k^2\equiv\pm 1 \pmod{n/2}$, or $n=4k$. In this note we prove that this conjecture is true. At the same time we determine all possible isomorphisms among the generalized Petersen graphs, the canonical double covers of the generalized Petersen graphs, and the double generalized Petersen graphs. Based on these we completely determine the full automorphism group of the canonical double cover of $\GP(n,k)$ for any pair of integers $n, k$ with $1 \leqslant k < n/2$.


\textit{Key words:} canonical double cover; stable graph; generalized Petersen graph; double generalized Petersen graph

\textit{Mathematics Subject Classifications:} 05C25,20B25
\end{abstract}

\section{Introduction}

All graphs considered in the note are finite, simple and undirected. As usual, for a graph $\Gamma$ we use $V(\Gamma)$, $E(\Gamma)$ and $\Aut(\Gamma)$ to denote its vertex set, edge set and automorphism group, respectively. For a positive integer $n$, denote by $\ZZ_n$, $\D_{2n}$, $\A_n$ and $\Sy_n$ the cyclic group of order $n$, the dihedral group of order $2n$, the alternating group of degree $n$ and the symmetric group of degree $n$, respectively.

The \emph{canonical double cover} of a graph $\Gamma$ (see, for example,~\cite{{LMS2015}}), denoted by $\D(\Gamma)$, is defined to be the direct product of $\Gamma$ and $K_2$, where $K_2$ is the complete graph of order $2$. That is, $\D(\Gamma)$ is the graph with vertex set $V(\Gamma) \times \ZZ_2$ in which $(u, x)$ and $(v, y)$ are adjacent if and only if $u$ and $v$ are adjacent in $\Gamma$ and $x \ne y$. In the literature $\D(\Gamma)$ is also called \cite{KP2018} the Kronecker cover of $\Gamma$. It can be verified that $\D(\Gamma)$ is connected if and only if $\Gamma$ is connected and non-bipartite (see, for example, \cite[Theorem~3.4]{BHM1980}). Clearly,
\[
\Aut(\D(\Gamma))\gtrsim \Aut(\Gamma)\times\Aut(K_2)\cong\Aut(\Gamma)\times\ZZ_2,
\]
where $X \gtrsim Y$ means that $X$ contains a subgroup that is isomorphic to $Y$. If $\Aut(\D(\Gamma))\cong\Aut(\Gamma)\times\ZZ_2$, then $\Gamma$ is called \emph{stable}; otherwise, $\Gamma$ is called \emph{unstable}. It can be easily verified (see, for example,~\cite[Proposition~4.1]{MSZ1989}) that a graph is unstable if it is disconnected, or bipartite with nontrivial automorphism group, or contains two distinct vertices with the same neighborhood. In light of this observation, we call an unstable graph \emph{nontrivially unstable} if it is connected, non-bipartite and vertex-determining, and \emph{trivially unstable} otherwise, where a graph is said to be \emph{vertex-determining} if no two vertices have the same neighborhood in the graph.

The stability of graphs was first studied in~\cite{MSZ1989} by Maru\v{s}i\v{c}, Scapellato and Zagaglia Salvi using the language of symmetric $(0,1)$ matrices. Since then this concept has been studied extensively by several authors from different viewpoints~\cite{LMS2015,MSZ1992,NS1996,Surowski2001,Surowski2003,Wilson2008}. In \cite{NS1996}, the stability of graphs played an important role in finding regular embeddings of canonical double covers on orientable surfaces. In \cite{LMS2015}, close connections between the stability and two-fold automorphisms of graphs were found. In \cite{MSZ1992}, searching for nontrivially unstable graphs led to the introduction of generalized Cayley graphs, and it was proved among others that every generalized Cayley graph which is not a Cayley graph is unstable. In \cite{Surowski2001}, methods for constructing arc-transitive unstable graphs were given, and three infinite families of such graphs were constructed as applications. Stability of circulant graphs was studied in \cite{Wilson2008} by Wilson and in \cite{QXZ} by the authors of the present paper, where in the latter paper an open question in \cite{Wilson2008} about the stability of arc-transitive circulant graphs was answered and an infinite family of counterexamples to a conjecture of  Maru\v{s}i\v{c}, Scapellato and Zagaglia Salvi \cite{MSZ1989} was constructed.

Apart from circulant graphs, Wilson \cite{Wilson2008} also studied the stability of a few other interesting families of graphs, notably generalized Petersen graphs. Given integers $n$ and $k$ with $1 \leqslant k < n/2$, the \emph{generalized Petersen graph} $\GP(n,k)$ is the cubic graph with $2n$ vertices, say, $u_0, \ldots, u_{n-1}, v_0,  \ldots, v_{n-1}$, and edges $\{u_i, u_{i+1}\}$, $\{u_i, v_i\}$, $\{v_i, v_{i+k}\}$, for $i \in \{0, \ldots, n-1\}$, with subscripts modulo $n$. It is readily seen that $\GP(5,2)$ is the well-known Petersen graph. It is also easy to see that $\GP(n,k)$ is connected and vertex-determining. Beginning with  \cite{Watkins1969}, generalized Petersen graphs have been studied widely in many different contexts. In particular, in \cite[Theorems~P.1--P.2]{Wilson2008}, Wilson proved that  $\GP(n, k)$ is unstable provided that $(n,k)$ satisfies one of the following conditions:
\begin{enumerate}[{\rm(P.1)}]
\item $n=2m$, where $m\geqslant 3$ is odd, and $k$ is even such that $k^2\equiv\pm1\pmod{m}$;
\item $n=4k$ and $k$ is even.
\end{enumerate}
In \cite[p.377]{Wilson2008}, Wilson conjectured that the converse of this statement is also true:
\begin{conjecture}\label{conj}
Every nontrivially unstable generalized Petersen graph $\GP(n,k)$ satisfies \emph{(P.1)} or \emph{(P.2)}.
\end{conjecture}

In this note we prove this conjecture through determining the automorphism groups of the canonical double covers of $\GP(n,k)$ for any integers $n, k$ with $1 \leqslant k < n/2$. Let
\[
\DGP(n,k) = \D(\GP(n,k))
\]
be the canonical double cover of $\GP(n,k)$ and
\[
A(n,k) = \Aut(\DGP(n,k))
\]
the full automorphism group of $\DGP(n,k)$. Obviously, $\DGP(n,k)$ is a cubic graph of order $4n$. To state our main result, we need to recall the following groups introduced in \cite{KP2016}.
\begin{equation}
\label{fn}
F(n) = \langle \rho,\delta \mid \rho^n=\delta^2=1, \delta\rho=\rho^{-1}\delta\rangle,
\end{equation}
\begin{equation}
\label{hnk}
H(n, k) = \langle \rho,\alpha \mid \rho^n=\alpha^4=1, \alpha\rho=\rho^{k}\alpha^{-1}\rangle,
\end{equation}
\begin{equation}
\label{jnk}
J(n,k) = \langle \rho,\delta,\alpha \mid \rho^n=\delta^2=\alpha^2=1, \delta\rho=\rho^{-1}\delta, \alpha\rho=\rho^k\alpha, \alpha\delta = \delta\alpha\rangle,
\end{equation}
\begin{align}
\label{knk}
K(n,k) = & \langle \rho, \delta, \beta \mid \rho^{n}=\delta^2=\beta^2=1, \delta\rho= \rho^{-1}\delta, \beta\rho=\rho\beta, \delta\beta=\beta\delta\rangle,
\end{align}
\begin{align}
\label{lnk}
L(n,k) = & \langle \rho, \delta, \beta, \psi \mid \rho^{n}=\delta^2=\beta^2=\psi^2=1, \delta\rho= \rho^{-1}\delta, \beta\rho=\rho\beta, \delta\beta=\beta\delta, \\ \non
          & \psi\rho=\rho^k\beta\psi, \psi\delta=\delta\psi, \psi\beta=\rho^\frac{n}{2}\psi \rangle,
\end{align}
\begin{align}
\label{mnk}
M(n,k) = & \langle \rho, \delta, \beta, \psi \mid \rho^{n}=\delta^2=\beta^2=\psi^4=1, \delta\rho= \rho^{-1}\delta, \beta\rho=\rho\beta, \delta\beta=\beta\delta, \\ \non
  & \psi\rho=\rho^k\beta\psi, \psi\delta=\delta\psi, \psi\beta=\rho^\frac{n}{2}\psi, \psi^2=\delta\rangle,
\end{align}
\begin{align}
\label{nnk}
N(n,k) = & \langle \rho, \delta, \beta, \eta \mid \rho^{n}=\delta^2=\beta^2=\eta^2=1, \delta\rho= \rho^{-1}\delta, \beta\rho=\rho\beta, \delta\beta=\beta\delta\rangle, \\ \non
 & \eta\rho=\rho\eta, \eta\delta=\delta\eta, \eta\beta=\rho^{\frac{n}{2}}\beta\eta\rangle.
\end{align}
Note that $F(n) \cong \D_{2n}$, $H(n, k) \cong \ZZ_{n}\rtimes\ZZ_4$, $J(n,k)  \cong \D_{2n}\rtimes\ZZ_2$ and $K(n,k) \cong \D_{2n}\times\ZZ_2$. Note also that $L(n,k)$, $M(n,k)$ and $N(n,k)$ are all semidirect products of $\D_{2n}\times\ZZ_2$ by $\ZZ_2$, but they are not necessarily isomorphic to each other.

The main result in this note is as follows.

\begin{theorem}
\label{A(n,k)}
Let $n$ and $k$ be integers with $1 \leqslant k < n/2$.

\begin{enumerate}[\rm (i)]
\item If both $n$ and $k$ are odd, then the following hold:
\begin{enumerate}
\item [\rm (i.1)]
if $k^2\not\equiv\pm1\pmod{n}$, then $A(n,k) = F(2n)$;
\item [\rm (i.2)]
if $k^2\equiv1\pmod{n}$, then $A(n,k) = J(2n, k)$;
\item [\rm (i.3)]
if $k^2\equiv-1\pmod{n}$, then $A(n,k) = H(2n, k)$.
\end{enumerate}

\item If $n$ is odd and $k$ is even, but $(n,k) \ne (5, 2)$, then the following hold:
\begin{enumerate}
\item [\rm (ii.1)]
if $k^2\not\equiv\pm1\pmod{n}$, then $A(n,k) = F(2n)$;
\item [\rm (ii.2)]
if $k^2\equiv1\pmod{n}$, then $A(n,k) = J(2n, n-k)$;
\item [\rm (ii.3)]
if $k^2\equiv-1\pmod{n}$, then $A(n,k) = H(2n, n-k)$.
\end{enumerate}
In addition,
\begin{enumerate}
\item [\rm (ii.4)] $A(5,2)\cong \Sy_5\times\ZZ_2$.
\end{enumerate}

\item If $n$ is even and $k$ is odd, but $(n,k) \ne (4, 1), (8, 3), (10, 3), (12, 5), (24, 5)$, then the following hold:
\begin{enumerate}
\item [\rm (iii.1)]
if $k^2\not\equiv\pm1\pmod{n}$, then $A(n,k) = F(n) \wr \Sy_2$;
\item [\rm (iii.2)]
if $k^2\equiv1\pmod{n}$, then $A(n,k) = J(n,k) \wr \Sy_2$;
\item [\rm (iii.3)]
if $k^2\equiv-1\pmod{n}$, then $A(n,k) = H(n, k)\wr \Sy_2$.
\end{enumerate}
In addition, we have
\begin{enumerate}
\item [\rm (iii.4)] $A(4,1)\cong (\Sy_4\times \ZZ_2)\wr \Sy_2$;
\item [\rm (iii.5)] $A(8,3)\cong(\GL(2,3)\rtimes \ZZ_2)\wr \Sy_2$;
\item [\rm (iii.6)] $A(10,3)\cong (\Sy_5\times\ZZ_2)\wr \Sy_2$;
\item [\rm (iii.7)] $A(12,5)\cong (\Sy_4\times \Sy_3)\wr \Sy_2$;
\item [\rm (iii.8)] $A(24,5)\cong((\GL(2,3)\times\ZZ_3)\rtimes \ZZ_2)\wr \Sy_2$.
\end{enumerate}

\item If both $n$ and $k$ are even, but $(n, k) \ne (10,2)$, then the following hold:
\begin{enumerate}
\item [\rm (iv.1)] if $k^2 \equiv 1 \pmod{n/2}$, then $A(n,k) = L(n,k)$;
\item [\rm (iv.2)] if $k^2 \equiv -1 \pmod{n/2}$, then $A(n,k) = M(n,k)$;
\item [\rm (iv.3)] if $n=4k$, then $A(n,k) = N(n,k)$;
\item [\rm (iv.4)] if $k^2 \not\equiv\pm1 \pmod{n/2}$ and $n\neq4k$, then $A(n,k) = K(n,k)$;
\end{enumerate}
In addition,
\begin{enumerate}
\item [\rm (iv.5)] $A(10,2)\cong (\A_5\times\ZZ_2^2)\rtimes\ZZ_2$.
\end{enumerate}
\end{enumerate}
\end{theorem}

The following corollary of Theorem \ref{A(n,k)} settles Conjecture \ref{conj} affirmatively.

\begin{corollary}
\label{thm:stab}
Let $n$ and $k$ be integers with $1 \leqslant k < n/2$.
\begin{enumerate}[\rm (i)]
\item If $n$ is odd, then $\GP(n,k)$ is stable.
\item $\GP(n,k)$ is trivially unstable if and only if $n$ is even and $k$ is odd.
\item If both $n$ and $k$ are even, then $\GP(n,k)$ is nontrivially unstable if and only if one of the following holds:
\begin{enumerate}
\item[\rm (iii.1)] $k^2\equiv\pm 1 \pmod{n/2}$;
\item[\rm (iii.2)] $n=4k$.
\end{enumerate}
\end{enumerate}
\end{corollary}


We would like to emphasize that Theorem~\ref{A(n,k)} contains more information than needed to prove Corollary~\ref{thm:stab}. In general, it is challenging to determine the full automorphism group of a graph. An early success in this line of research is the determination of the automorphism group of $\GP(n,k)$ achieved by Frucht, Graver and Watkins in~\cite{FGW1971}, and Theorem~\ref{A(n,k)} gives parallel results for $\DGP(n,k)$. In a recent paper~\cite{KP2018}, Krnc and Pisanski characterized all generalized Petersen graphs which are isomorphic to the canonical double covers of some graphs, and they adverted~\cite[p.16]{KP2018} that it would be interesting to investigate the canonical double covers of generalized Petersen graphs. It is envisaged that Theorem~\ref{A(n,k)} may be useful in studying some problems for $\DGP(n,k)$, especially those involving symmetries of this graph.

In a previous version of the present paper (see \url{http://arxiv.org/abs/1807.07228v1}), we proved Theorem~\ref{A(n,k)} using similar methodologies as in \cite{FGW1971}, the most technical part being determining $A(n,k)$ when both $n$ and $k$ are even. Very recently, we found that we can give a shorter proof of Theorem~\ref{A(n,k)}, as presented in the current version, by linking $\DGP(n, t)$ to another double cover of $\GP(n, t)$ which was introduced by Zhou and Feng in \cite{ZF2012}, where $1 \leqslant t < n/2$. This double cover of $\GP(n,t)$, denoted by $\DP(n,t)$ and called a \emph{double generalized Petersen graph} \cite{ZF2012}, is defined  to have vertex set
\[
\{x_0, \ldots, x_{n-1}, y_0,  \ldots, y_{n-1},\overline{x}_0,  \ldots, \overline{x}_{n-1}, \overline{y}_0,  \ldots, \overline{y}_{n-1}\}
\]
and edge set
\[
\{\{x_i, x_{i+1}\}, \{y_i, y_{i+1}\}, \{x_i, \overline{x}_i\}, \{y_i, \overline{y}_i\}, \{\overline{x}_i, \overline{y}_{i+t}\}, \{\overline{y}_i, \overline{x}_{i+t}\} \mid i\in\{0,\ldots, n-1\}\},
\]
with subscripts modulo $n$. In~\cite{KP2016}, Kutnar and Petecki determined several permutations of $V(\DP(n,t))$ and proved that they generate the automorphism group of $\DP(n,t)$. Our shorter proof of Theorem~\ref{A(n,k)} is achieved through determining all possible isomorphisms between $\DGP(n,k)$ and $\DP(n,t)$. In fact, we can determine all possible isomorphisms among $\DGP(n,k)$, $\GP(2n,s)$ and $\DP(n,t)$ as shown in the following theorem.  

\begin{theorem}\label{3isomorphism}
Let $n$, $k$, $s$ and $t$ be integers with $1 \leqslant k< n/2$, $1\leqslant s<n$ and $1 \leqslant t< n/2$.  
\begin{enumerate}[{\rm (i)}]
\item $\DGP(n,k)\cong\GP(2n,s)$ for some integer $s$ with $1\leqslant s<n$ if and only if $n$ is odd. Moreover, if $n$ and $k$ are both odd, then $\DGP(n,k)\cong\GP(2n,k)$; if $n$ is $odd$ and $k$ is even, then $\DGP(n,k)\cong\GP(2n,n-k)$.
\item $\DGP(n,k)\cong\DP(n,t)$ for some integer $t$ with $1 \leqslant t< n/2$ if and only if $n$ and $k$ are both even. Moreover, if $n$ and $k$ are both even, then $\DGP(n,k)\cong\DP(n,k)$.
\item $\DP(n,t)\cong\GP(2n,s)$ for some integer $s$ with $1\leqslant s<n$ if and only if $n$ is odd and $\gcd(n,t)=1$. Moreover, if $n$ is odd and $\gcd(n,t)=1$, then $\DP(n,t)\cong \GP(2n,s)$, where $s$ is the unique even integer such that $1\leqslant s<n$ and $st\equiv\pm 1 \pmod{n}$.
\item It can not happen that $\DGP(n,k)\cong\GP(2n,s)\cong\DP(n,t)$ for any pair of integers $s$ and $t$ with $1\leqslant s<n$ and $1\leqslant t<n/2$. 
\end{enumerate}
\end{theorem}

After setting up notation and recalling a few known results on generalized Petersen graphs in the next section, we prove Theorem~\ref{3isomorphism} in Section~\ref{DP}. As shown in part (ii) of Theorem ~\ref{3isomorphism}, $\DGP(n,k)\cong\DP(n,k)$ with $n$ and $k$ even are the only isomorphisms between the canonical double covers of generalized Petersen graphs and double generalized Petersen graphs. Using these isomorphisms and some results in \cite{KP2016}, we prove Theorem~\ref{A(n,k)} and then Corollary~\ref{thm:stab} in Section~\ref{stability}.

\section{Preliminaries}
\label{preli}

We will use the following notation throughout the note. Let $n$ and $k$ be integers with $1 \leqslant k< n/2$. As before we label the vertices of $\GP(n,k)$ by
\[
u_0, u_1, \ldots, u_{n-1}, v_0, v_1, \ldots, v_{n-1}
\]
in such a way that the edges of $\GP(n,k)$ are given by
\[
\{u_i, u_{i+1}\},\; \{u_i, v_i\},\; \{v_i, v_{i+k}\},\; i\in\{0,1,\ldots, n-1\},
\]
with subscripts modulo $n$. Then the vertex set of $\DGP(n,k)$ is
\begin{align*}
V(\DGP(n,k))=\{ & (u_0, 0), (u_1,0), \ldots, (u_{n-1},0), (u_0,1), (u_1,1), \ldots, (u_{n-1},1),\\
& (v_0, 0), (v_1,0), \ldots, (v_{n-1},0), (v_0,1), (v_1,1), \ldots, (v_{n-1},1)\}
\end{align*}
and the edge set of $\DGP(n,k)$ consists of
\begin{equation}\label{edge2}
\{(u_i, j), (u_{i+1}, 1-j)\},\; \{(u_i, j), (v_i, 1-j)\},\; \{(v_i, j), (v_{i+k}, 1-j)\}
\end{equation}
for $i\in\{0,1,\ldots, n-1\}$ and $j\in\{0, 1\}$, with subscripts taken modulo $n$.

The automorphism group of $\GP(n,k)$ was determined by Frucht, Graver and Watkins (see~\cite[Theorems 1 and 2, p.217--218]{FGW1971}). We present their result in the following lemma, where the groups $F(n)$, $J(n,k)$ and $H(n,k)$ are as defined in \eqref{fn}, \eqref{jnk} and \eqref{hnk}, respectively. 

\begin{lemma}\label{A}
Let $n$ and $k$ be integers with $1 \leqslant k< n/2$. If $(n,k)\neq(4,1)$, $(5,2)$, $(8,3)$, $(10,2)$, $(10,3)$, $(12,5)$, $(24,5)$, then the following hold:
\begin{enumerate}[{\rm(i)}]
\item if $k^2\not\equiv\pm1\pmod{n}$, then $\Aut(\GP(n,k)) = F(n)$;
\item if $k^2\equiv1\pmod{n}$, then $\Aut(\GP(n,k)) = J(n,k)$;
\item if $k^2\equiv-1\pmod{n}$, then $\Aut(\GP(n,k)) = H(n,k)$.
\end{enumerate}
Moreover, the following hold:
\begin{enumerate}[{\rm(iv)}]
\item $\Aut(\GP(4,1))\cong \Sy_4\times \ZZ_2$; \smallskip
\item[{\rm(v)}] $\Aut(\GP(5,2))\cong \Sy_5$; \smallskip
\item[{\rm(vi)}] $\Aut(\GP(8,3))=X\cong\GL(2,3)\rtimes \ZZ_2$, where
\[
X=\langle \rho, \delta, \sigma\mid \rho^8=\delta^2=\sigma^3=1, \delta\rho\delta=\rho^{-1}, \delta\sigma\delta=\sigma^{-1}, \sigma\rho\sigma=\rho^{-1}, \sigma\rho^4=\rho^4\sigma\rangle;
\]
\item[{\rm(vii)}] $\Aut(\GP(10,2))\cong \A_5\times\ZZ_2$; \smallskip
\item[{\rm(viii)}] $\Aut(\GP(10,3))\cong \Sy_5\times \ZZ_2$; \smallskip
\item[{\rm(ix)}] $\Aut(\GP(12,5))=X\cong \Sy_4\times \Sy_3$, where
\[
X=\langle \rho, \delta, \sigma\mid \rho^{12}=\delta^2=\sigma^3=1, \delta\rho\delta=\rho^{-1}, \delta\sigma\delta=\sigma^{-1}, \sigma\rho\sigma=\rho^{-1}, \sigma\rho^4=\rho^4\sigma\rangle;
\]
\item[{\rm(x)}] $\Aut(\GP(24,5))=X\cong(\GL(2,3)\times\ZZ_3)\rtimes \ZZ_2$, where
\[
X=\langle \rho, \delta, \sigma\mid (\sigma\rho)^2=\delta^2=\sigma^3=1, \delta\rho\delta=\rho^{-1}, \delta\sigma\delta=\sigma^{-1}, \sigma\rho^4=\rho^4\sigma\rangle.
\]
\end{enumerate}
\end{lemma}

%
%

The next lemma, as a special case of~\cite[Proposition~9]{BPZ2005}, gives all possible isomorphisms between generalized Petersen graphs.

\begin{lemma}\label{isomorphism}
Let $n$, $r$ and $s$ be integers with $1 \leqslant r< n/2$, $1 \leqslant s< n/2$ and $r\neq s$. Then $\GP(n,r)$ is isomorphic to $\GP(n,s)$ if and only if $rs\equiv\pm 1 \pmod{n}$.
\end{lemma}

\section{Isomorphisms among $\DGP(n,k)$, $\GP(2n,s)$ and $\DP(n,t)$}\label{DP}

First we give the isomorphisms between $\DGP(n,k)$ and $\GP(n,t)$ for odd $n$. Following~\cite{KP2016}, we call edges of $\DP(n,t)$ in
\begin{align*}
& \{\{x_i, x_{i+1}\}, \{y_i, y_{i+1}\} \mid i\in\{0,\ldots, n-1\}\},\\
& \{\{x_i, \overline{x}_i\}, \{y_i, \overline{y}_i\} \mid i\in\{0,\ldots, n-1\}\}
\end{align*}
and
\[
\{\{\overline{x}_i, \overline{y}_{i+t}\}, \{\overline{y}_i, \overline{x}_{i+t}\} \mid i\in\{0,\ldots, n-1\}\}
\]
the \emph{outer edges}, \emph{spokes} and \emph{inner edges} of $\DP(n,t)$, respectively. The first two parts of the following proposition can be found in~\cite[Proposition 12]{KP2018}, and the third part is true as $\DP(n,t)$ contains an $n$-cycle while the bipartite graphs $\DGP(n,k)$ does not.

\begin{proposition}\label{ncong1}
Let $n$, $k$ and $t$ be integers with $n$ odd, $1 \leqslant k< n/2$ and $1 \leqslant t< n/2$. Then the following hold:
\begin{enumerate}[{\rm (i)}]
\item if $k$ is odd, then $\DGP(n,k)\cong\GP(2n,k)$;
\item if $k$ is even, then $\DGP(n,k)\cong\GP(2n,n-k)$;
\item $\DGP(n,k)$ is not isomorphic to $\DP(n,t)$.
\end{enumerate}
\end{proposition}




For a positive integer $m$ and a graph $\Gamma$, denote by $m\Gamma$ the graph consisting of $m$ vertex-disjoint copies of $\Gamma$. Note that, for even $n$ and odd $k$, since $\GP(n,k)$ is bipartite (see, for example,~\cite[Proposition 4.3]{AL2009} or~\cite[Theorem 2]{BPZ2005}), the canonical double cover $\DGP(n,k)$ is isomorphic to $2\GP(n, k)$. Thus we have the following lemma.

\begin{lemma}\label{ncong2}
Let $n$ and $k$ be integers with $n$ even, $k$ odd and $1 \leqslant k< n/2$. Then $\DGP(n,k)\cong2\GP(n, k)$. In particular, $\DGP(n,k)$ is not isomorphic to $\DP(n,t)$ for any integer $t$ with $1 \leqslant t< n/2$. 
\end{lemma}


The next Lemma can be easily proved using Proposition~\ref{ncong1}(iii), Lemma~\ref{ncong2} and the observation that the mapping
\begin{align*}
&(u_i,0)\mapsto x_i,\quad(u_i,1)\mapsto y_i,\quad(v_i,1)\mapsto\overline{x}_i,\quad(v_i,0)\mapsto\overline{y}_i,\\
&(u_j,1)\mapsto x_j,\quad(u_j,0)\mapsto y_j,\quad(v_j,0)\mapsto\overline{x}_j,\quad(v_j,1)\mapsto\overline{y}_j
\end{align*}
for $i\in\{0,2,\ldots,n-2\}$ and $j\in\{1,3,\ldots,n-1\}$ gives an isomorphism from $\DGP(n,k)$ to $\DP(n,k)$.

\begin{lemma}\label{isomorphic}
Let $n$ and $k$ be integers with $1 \leqslant k< n/2$. Then $\DGP(n,k)\cong\DP(n,t)$ for some integer $t$ with $1 \leqslant t< n/2$ if and only if $n$ and $k$ are both even. Moreover, if $n$ and $k$ are both even, then $\DGP(n,k)\cong\DP(n,k)$.
\end{lemma}



\begin{lemma}\label{isomorphic2}
Let $n$, $k$ and $t$ be integers with $1 \leqslant k< n/2$ and $1 \leqslant t< n/2$. Then the following hold:
\begin{enumerate}[{\rm (i)}]
\item $\DP(n,t)$ is isomorphic to a generalized Petersen graph if and only if $n$ is odd and $\gcd(n,t)=1$;
\item if $n$ is odd and $\gcd(n,t)=1$, then $\DP(n,t)\cong \GP(2n,s)$, where $s$ is the unique even integer such that $1\leqslant s<n$ and $st\equiv\pm 1 \pmod{n}$.
\end{enumerate}
\end{lemma}

\begin{proof}
First assume that $\DP(n,t)\cong\GP(m,r)$ for some integers $m$ and $r$ with $1 \leqslant r<m/2$. Then $m=2n$ and $(u_0,u_1,\ldots,u_{2n-1})$ is a cycle in $\GP(m,r)$, and so there is a cycle $C$ of length $2n$ in $\DP(n,t)$ corresponding to $(u_0,u_1,\ldots,u_{2n-1})$. Clearly, the outer edges of $\DP(n,t)$ form two vertex-disjoint cycles of length $n$. It follows that $C$ either consists of inner edges only or consists of outer edges, spokes and inner edges. If the former occours, then $C$ is of the form $(\overline{x}_0, \overline{y}_t, \overline{x}_{2t}, \overline{y}_{2t}, \cdots, \overline{x}_{(n-2)t}, \overline{y}_{(n-1)t})$, and so $n$ is odd and $\gcd(n,t)=1$. Suppose that the latter occurs. Note that for any two edges in $(u_0,u_1,\ldots,u_{2n-1})$, there exists an element  in $\Aut(\GP(m,r))$ which maps one edge to the other. This implies that there exists $\pi\in\Aut(\DP(n,t))$ which maps some spoke to an edge that is not a spoke. Thereby we derive from~\cite[Lemma~3.6]{KP2016} that $\DP(n,t)$ is edge-transitive, and so $\GP(m,r)$ is edge-transitive. Then by \cite[p.~212]{FGW1971} we have $(m,r)=(4,1)$, $(5,2)$, $(8,3)$, $(10,2)$, $(10,3)$, $(12,5)$ or $(24,5)$. However, computation in \textsc{Magma}~\cite{magma} shows that $\GP(4,1)$, $\GP(5,2)$, $\GP(8,3)$, $\GP(12,5)$ and $\GP(24,5)$ are not isomorphic to any double cover of any generalized Petersen graph. Thus $m=10$ and $n=5$, whence $n$ is odd and $\gcd(n,t)=1$.

Conversely, assume that $n$ is odd and $\gcd(n,t)=1$. Then there exists an unique even integer $s$ such that $1\leqslant s<n$ and $st\equiv\pm1\pmod{n}$.
It can be verified that the mapping
\begin{align*}
&u_i\mapsto
\begin{cases}
\overline{x}_{it}\quad&\text{if $i$ is even and $i<n$}\\
\overline{y}_{it}\quad&\text{if $i$ is odd and $i<n$}\\
\overline{x}_{(i-n)t}\quad&\text{if $i$ is even and $i\geqslant n$}\\
\overline{y}_{(i-n)t}\quad&\text{if $i$ is odd and $i\geqslant n$}
\end{cases}\\
&v_i\mapsto
\begin{cases}
x_{it}\quad&\text{if $i$ is even and $i<n$}\\
y_{it}\quad&\text{if $i$ is odd and $i<n$}\\
x_{(i-n)t}\quad&\text{if $i$ is even and $i\geqslant n$}\\
y_{(i-n)t}\quad&\text{if $i$ is odd and $i\geqslant n$}
\end{cases}
\end{align*}
for $i\in\{0,\ldots, 2n-1\}$ defines an isomorphism from  $\GP(2n,s)$ to $\DP(n,t)$. Hence $\DP(n,t)\cong \GP(2n,s)$. This completes the proof of statements (i) and (ii).
\end{proof}

We conclude this section by proving Theorem~\ref{3isomorphism}.

\begin{proof}
By Proposition~\ref{ncong1}(i) and~\cite[Corollary 20]{KP2018} the statements in part~(i) hold. By Lemmas~\ref{isomorphic} and~\ref{isomorphic2} we obtain the statements in parts~(ii) and~(iii), respectively. The statements in part~(iv) of follows from the statements in parts~(i),(ii) and~(iii). 
\end{proof}

\section{Proofs of Theorem~\ref{A(n,k)} and Corollary \ref{thm:stab}}\label{stability}

In this section we determine $A(n,k)$ and the stability of generalized Petersen graphs.

\begin{proposition}\label{odd}
Let $n$ and $k$ be integers with $n$ odd and $1 \leqslant k< n/2$. Then $\GP(n,k)$ is stable and $A(n,k)$ is given as follows:
\begin{enumerate}[\rm (i)]
\item If $k$ is odd, then the following hold:
\begin{enumerate}
\item [\rm (i.1)]
if $k^2\not\equiv\pm1\pmod{n}$, then $A(n,k) = F(2n)$;
\item [\rm (i.2)]
if $k^2\equiv1\pmod{n}$, then $A(n,k) = J(2n, k)$;
\item [\rm (i.3)]
if $k^2\equiv-1\pmod{n}$, then $A(n,k) = H(2n, k)$.
\end{enumerate}
\item If $k$ is even and $(n,k) \ne (5, 2)$, then the following hold:
\begin{enumerate}
\item [\rm (ii.1)]
if $k^2\not\equiv\pm1\pmod{n}$, then $A(n,k) = F(2n)$;
\item [\rm (ii.2)]
if $k^2\equiv1\pmod{n}$, then $A(n,k) = J(2n, n-k)$;
\item [\rm (ii.3)]
if $k^2\equiv-1\pmod{n}$, then $A(n,k) = H(2n, n-k)$.
\end{enumerate}
In addition,
\begin{enumerate}
\item [\rm (ii.4)] $A(5,2)\cong \Sy_5\times\ZZ_2$.
\end{enumerate}
\end{enumerate}
\end{proposition}

\begin{proof}
First assume that $k$ is odd. Then by Proposition~\ref{ncong1}(i) we have $\DGP(n,k)\cong\GP(2n,k)$. If $k^2\not\equiv\pm1\pmod{n}$, then $k^2\not\equiv\pm1\pmod{2n}$, and hence we obtain from Lemma~\ref{A} that $\Aut(\GP(n,k)) = F(n) \cong \D_{2n}$ and $A(n,k) = F(2n) \cong \D_{4n}$. Thus $|A(n,k)|=2|\Aut(\GP(n,k))|$, which shows that $\GP(n,k)$ is stable. 

Similarly, we can prove other parts of the proposition using Proposition~\ref{ncong1} and Lemma~\ref{A}.
\end{proof}

\begin{proposition}\label{even}
Let $n$ and $k$ be even integers with $1 \leqslant k < n/2$. If $(n,k)\neq(10,2)$, then the following hold:
\begin{enumerate}[{\rm(i)}]
\item if $k^2\equiv1\pmod{n/2}$, then $A(n,k) = L(n,k)$;
\item if $k^2\equiv-1\pmod{n/2}$, then $A(n,k) = M(n,k)$;
\item if $n = 4k$, then $A(n,k) = N(n,k)$;
\item if $k^2 \not \equiv \pm 1 \pmod{n/2}$ and $n \ne 4k$, then $A(n,k) = K(n,k)$.
\end{enumerate}
Moreover, $A(10, 2)\cong (\A_5\times\ZZ_2^2)\rtimes\ZZ_2$
\end{proposition}

\begin{proof}
Since both $n$ and $k$ are even,  by Lemma~\ref{isomorphic} we have $\DGP(n,k)\cong\DP(n,k)$.

Assume $k^2\equiv1\pmod{n/2}$. From~\cite[Propositions~3.1,~3.8,~Corollary~3.11]{KP2016} and the proof of~\cite[Proposition~3.4]{KP2016} we see that
$A(n,k) = \langle \rho, \delta, \beta, \psi \rangle$ with $|A(n,k)|=8n$ and
\[
\delta\rho= \rho^{-1}\delta, \quad \beta\rho=\rho\beta, \quad \delta\beta=\beta\delta, \quad \psi\rho=\rho^k\beta\psi, \quad \psi\delta=\delta\psi, \quad \psi\beta=\rho^\frac{n}{2}\psi,
\]
where the generators $\rho$, $\beta$, $\delta$ and $\psi$ are the permutations $\alpha$, $\beta$, $\gamma$ and $\psi$ defined in~\cite[p.2863]{KP2016}, respectively. By the definition of these permutations it is easy to verify that $\rho^{n}=\delta^2=\beta^2=\psi^2=1$. Since $|L(n,k)|=8n=|A(n,k)|$, we then conclude that $A(n,k) = L(n,k)$, proving statement~(i).

Assume $k^2\equiv-1\pmod{n/2}$. From~\cite[Propositions~3.1,~3.8,~Corollary~3.11]{KP2016} and the proof of~\cite[Proposition~3.4]{KP2016} we see that
$A(n,k) = \langle \rho, \delta, \beta, \psi \rangle$ with $|A(n,k)|=8n$ and
\[
\delta\rho= \rho^{-1}\delta, \quad \beta\rho=\rho\beta, \quad \delta\beta=\beta\delta, \quad \psi\rho=\rho^k\beta\psi, \quad \psi\delta=\delta\psi, \quad \psi\beta=\rho^\frac{n}{2}\psi, \quad \psi^2=\delta,
\]
where the generators $\rho$, $\beta$, $\delta$ and $\psi$ are the permutations $\alpha$, $\beta$, $\gamma$ and $\psi$ defined in~\cite[p.2863]{KP2016}, respectively. By the definition of these permutations it is direct to verify that $\rho^{n}=\delta^2=\beta^2=\psi^4=1$. Since $|M(n,k)|=8n=|A(n,k)|$, it follows that $A(n,k) = M(n,k)$, as statement~(ii) asserts.

Assume $n = 4k$. Then from~\cite[Propositions~3.1,~3.8,~Corollary~3.11]{KP2016} and the proof of~\cite[Proposition~3.4]{KP2016} we see that
$A(n,k) = \langle \rho, \delta, \beta, \eta \rangle$ with $|A(n,k)|=8n$ and
\[
\delta\rho= \rho^{-1}\delta, \quad \beta\rho=\rho\beta, \quad \delta\beta=\beta\delta, \quad \eta\rho=\rho\eta, \quad \eta\delta=\delta\eta, \quad \eta\beta=\rho^{\frac{n}{2}}\beta\eta,
\]
where the generators $\rho$, $\beta$, $\delta$ and $\eta$ are the permutations $\alpha$, $\beta$, $\gamma$ and $\eta$ defined in~\cite[p.2863]{KP2016}, respectively. By the definition of these permutations it is easy to verify that $\rho^{n}=\delta^2=\beta^2=\eta^2=1$. As $|N(n,k)|=8n=|A(n,k)|$, we obtain that $A(n,k) = N(n,k)$, proving statement~(iii).

Now assume $k^2 \not \equiv \pm 1 \pmod{n/2}$ and $n \ne 4k$. From~\cite[Propositions~3.1,~3.8,~Corollary~3.11]{KP2016} and the proof of~\cite[Proposition~3.4]{KP2016} we see that
$A(n,k) = \langle \rho, \delta, \beta\rangle$ with $|A(n,k)|=4n$ and
\[
\delta\rho= \rho^{-1}\delta, \quad \beta\rho=\rho\beta, \quad \delta\beta=\beta\delta,
\]
where the generators $\rho$, $\beta$ and $\delta$ are the permutations $\alpha$, $\beta$ and $\gamma$ defined in~\cite[p.2863]{KP2016}, respectively. Moreover, it is readily seen from the definition of these permutations that $\rho^{n}=\delta^2=\beta^2=1$. Since $|K(n,k)|=4n=|A(n,k)|$, it follows that $A(n,k) = K(n,k)$, as statement~(iv) asserts.

Finally, computation in \textsc{Magma}~\cite{magma} shows that $A(10, 2)\cong (\A_5\times\ZZ_2^2)\rtimes\ZZ_2$. The proof is thus completed.
\end{proof}

We are now ready to prove Theorem~\ref{A(n,k)}:

\begin{proof}
If $n$ is odd, then Proposition~\ref{odd} shows that parts~(i) and~(ii) of Theorem~\ref{A(n,k)} hold. If $n$ is even and $k$ is odd, then by Lemma~\ref{ncong2} we have $\DGP(n,k)\cong2\GP(n, k)$, and hence $A(n,k)\cong\Aut(\GP(n,k))\wr\Sy_2$ (see~\cite{Sabidussi1959}), which together with Lemma~\ref{A} leads to part~(iii) of Theorem~\ref{A(n,k)}. If both $n$ and $k$ are even, then from Proposition~\ref{even} we obtain part~(iv) of Theorem~\ref{A(n,k)}. This completes the proof.
\end{proof}

We conclude the note by proving Corollary \ref{thm:stab}: 

\begin{proof}
If $n$ is odd, then according to Proposition~\ref{odd}, $\GP(n, k)$ is stable. Since $\GP(n,k)$ is connected and vertex-determining, it is trivially unstable if and only if it is bipartite. Recall that $\GP(n,k)$ is bipartite if and only if $n$ is even and $k$ is odd (see, for example,~\cite[Proposition 4.3]{AL2009} or~\cite[Theorem 2]{BPZ2005}). Thus $\GP(n,k)$ is trivially unstable if and only if $n$ is even and $k$ is odd. Now assume that both $n$ and $k$ are even. Then $k^2\not\equiv\pm1\pmod{n}$, and hence Lemma~\ref{A} implies that
\[
|\Aut(\GP(n,k))|=
\begin{cases}
2n\quad&\text{if $(n,k)\neq(10,2)$}\\
120\quad&\text{if $(n,k)=(10,2)$}.
\end{cases}
\]
Moreover, from Theorem~\ref{A(n,k)} we see that
\[
|A(n,k)|=
\begin{cases}
8n\quad&\text{if $(n,k)\neq(10,2)$ and either $k^2 \equiv\pm1\pmod{n/2}$ or $n = 4k$}\\
4n\quad&\text{if $k^2 \not \equiv\pm1\pmod{n/2}$ and $n \neq 4k$}\\
480\quad&\text{if $(n,k)=(10,2)$}.
\end{cases}
\]
Note that $(n,k)=(10,2)$ satisfies $k^2 \equiv -1 \pmod{n/2}$. It follows that $|A(n,k)|\neq2|\Aut(\GP(n,k))|$ if and only if $k^2 \equiv\pm1\pmod{n/2}$ or $n = 4k$. This shows that $\GP(n, k)$ is unstable if and only if $k^2 \equiv\pm1\pmod{n/2}$ or $n = 4k$, as desired.
\end{proof}

{\noindent\textsc{Acknowledgements.}} We would like to thank the anonymous referees for their valuable comments. The first author was supported by the Fundamental Research Funds for Beijing Universities allocated to Capital University of Economics and Business (XRZ2020058). This work was done during a visit of the first author to The University of Melbourne. She would like to thank Beijing Jiaotong University for supporting this visit and National Natural Science Foundation of China (11671030) for its financial support during her PhD program. She is very grateful to Professor Jin-Xin Zhou for suggesting the research topic.

\end{document}